\documentclass{article}
\usepackage{mathtools}
\usepackage{amssymb, amsmath, amsthm, amscd}
\usepackage[dvips]{graphics}
\usepackage[utf8]{inputenc}
\usepackage[all,cmtip]{xy}
\usepackage{dsfont}
\usepackage{dsfont}
\usepackage{mathrsfs}  
\usepackage{bbm}
\usepackage{textcomp}
\usepackage{eurosym}
\usepackage{enumitem}
\usepackage{setspace}
\usepackage{color}
 \usepackage{setspace}
\usepackage{color}
\usepackage[usenames,dvipsnames,svgnames,table]{xcolor}
\usepackage[colorlinks=true,
            linkcolor=blue,
            urlcolor=blue,
            citecolor=blue]{hyperref}
\usepackage{filecontents}
\newtheorem{lem}{Lemma}[section]
\newtheorem{pro}[lem]{Proposition}
\newtheorem{defi}[lem]{Definition}
\newtheorem{def/not}[lem]{Definition/Notations}
\newtheorem{thm}[lem]{Theorem}
\newtheorem{mthm}[lem]{Main Theorem}

\newtheorem{cor}[lem]{Corollary}
\newtheorem{rem}[lem]{Remark}

\newcommand{\C}{\mathbb{C}}

\newcommand{\SH}{\mathcal{SH}}

\title{Maximal subextension of $m$-subharmonic functions}
\author{Hichame Amal\footnote{Department of Mathematics, Laboratory LaREAMI, Regional Center of Trades of Education and Training, Kenitra, Morocco hichameamal@hotmail.com}\; and Ayoub El-Gasmi\footnote{Ibn Tofail University, Faculty of Sciences, Department of mathematics, PO 242 Kenitra Morroco, ayoub.el-gasmi@uit.ac.ma}}
\date{\it{Dedicated to Professor Omar Alehyane on the occasion of his retirement}}
\begin{document}

\maketitle
\begin{abstract}
  In this paper, we prove that given a quasi-$m$-hyperconvex domain $\Omega \subset X$ in a compact K\"ahler manifold $(X, \omega)$, and a function $\varphi $ in the weighted energy class $\mathcal{E}_\chi^m(\Omega, \omega)$ with respect to a convex weight function $\chi : \mathbb{R} \to \mathbb{R}$, then there exists a maximal $\omega$-$m$-subharmonic subextension $\tilde{\varphi}$ to $X$ that preserves the weighted energy and satisfies good control properties for its Hessian measure $
\mathbf{1}_\Omega H_m(\tilde{\varphi}) \leq \mathbf{1}_\Omega H_m(\varphi)
$. \\

\end{abstract}
\section{Introduction}
The subextension problem for plurisubharmonic functions - extending a given function to a larger domain while controlling its singularities and Monge-Ampère mass - has been extensively studied since the seminal work of El Mir \cite{E} in 1980. Fundamental contributions by Cegrell and Zeriahi \cite{CZ} showed that functions in the class $\mathcal{F}(\Omega)$ admit subextensions to larger hyperconvex domains with controlled Monge–Ampère mass. These results were later extended to weighted energy classes $\mathcal{E}_\chi(\Omega)$ by \cite{Hi,B,HL}, and to boundary value problems in \cite{CH,H,AC}.

Regarding maximal subextensions of plurisubharmonic functions, Cegrell-Kołodziej-Zeriahi established in \cite{CKZ1} that any plurisubharmonic function $\varphi$ belonging to a suitable Cegrell class in a hyperconvex domain $\Omega \Subset \mathbb{C}^n$ admits a maximal subextension to $\mathbb{C}^n$ whose Monge–Ampère measure is well-defined and supported on the contact set $\{ \tilde{\varphi} = \varphi \} \cup \partial\Omega$. This result was extended to compact K\"ahler manifolds in \cite{CKZ2}, using the framework of $\omega$-plurisubharmonic functions developed in \cite{GZ}.

In the context of $m$-subharmonic functions, analogous subextension results were obtained by Hung \cite{VH} and Le Mau Hai-Dung \cite{MV}, who showed that functions in the class $\mathcal{F}_m(\Omega)$ admit subextensions preserving the complex Hessian measure. More recently, the theory of subextension and approximation for $m$-subharmonic functions with prescribed boundary values in $m$-hyperconvex domains in $\mathbb{C}^n$ has been developed in \cite{AE,P}.

Building on these foundations, and inspired by the  ideas developed in Cegrell–Kołodziej–Zeriahi \cite{CKZ2}, we investigate the subextension problem for the complex Hessian operator in the setting of weighted energy classes on compact K\"ahler manifolds. More precisely, given a quasi-$m$-hyperconvex domain $\Omega \subset X$ in a compact K\"ahler manifold $(X, \omega)$ and a function $\varphi \in \mathcal{E}_\chi^m(\Omega, \omega)$ belonging to a weighted energy class associated with a convex weight function $\chi : \mathbb{R} \to \mathbb{R}$, we study the existence of a maximal $\omega$-$m$-subharmonic subextension $\tilde{\varphi}$ to $X$ that preserves the weighted energy and satisfies sharp control over its complex Hessian measure.

Our main result gives a positive answer under natural volume and energy constraints:

\begin{mthm}
Let $\Omega \subset X$ be a quasi-$m$-hyperconvex domain with $\displaystyle\int_\Omega \omega^n < \int_X \omega^n$, $\chi : \mathbb{R} \to \mathbb{R}$ be a convex weight function and  $\varphi \in \mathcal{E}_\chi^m(\Omega, \omega)$ such that
$\displaystyle
\int_\Omega H_m(\varphi) \leq \int_X \omega^n.
$
Then the maximal $\omega$-$m$-subharmonic subextension $\tilde{\varphi}$ of $\varphi$ to $X$ exists and satisfies:
\begin{enumerate}
\item[(i)] $\tilde{\varphi} \in \mathcal{E}_\chi^m(X, \omega)$, and 
$\displaystyle
\int_X -\chi(\tilde{\varphi})  H_m(\tilde{\varphi}) \leq \int_\Omega -\chi(\varphi)  H_m(\varphi).
$
\item[(ii)] The Hessian measure satisfies 
$
\mathbf{1}_\Omega H_m(\tilde{\varphi}) \leq \mathbf{1}_\Omega H_m(\varphi)
$
in the sense of Borel measures on $X$.
\item[(iii)] The measure $H_m(\tilde{\varphi})$ is supported on the set 
$\{ \tilde{\varphi} = \varphi \} \cup \partial \Omega.
$
\end{enumerate}
\end{mthm}
In the last part, by analogy with the plurisubharmonic case, we introduce the $m$-Lelong class and study the subextension problem in the particular case where $X = \mathbb{P}^n$.

The structure of the paper is as follows: In Section 2, we recall basic definitions and properties of $\omega$-$m$-subharmonic functions on $(\Omega,\omega)$ and the weighted energy classes $\mathcal{E}_\chi^m(\Omega, \omega)$. In Section 3, we study the maximal subextension of $\omega$-$m$-subharmonic functions from $\Omega$ to $X$.

\section{Preliminaries}
\subsection{$\omega$-$m$-subharmonic functions on $\Omega\subset X$}
Let $(X,\omega)$ be a compact K\"ahler manifold of dimension $n$ and $\Omega\subset X$ be an open subset, and let $D\subset\Omega$ be an open subset contained in a local chart. Fix an integer $m$ such that $1 \leq m \leq n$.
\begin{defi}
A function $u \in C^2(D,\mathbb{R})$ defined on an open subset $D \subset \mathbb{C}^n$ is called \emph{$m$-subharmonic} ($m$-sh) with respect to a K\"ahler form $\omega$ if:
$$
(dd^c u)^k \wedge \omega^{n-k} \geq 0, \quad \forall k \in \{1,\dots,m\}.
$$
\end{defi}

\begin{defi}
A function $u \in L^1(D,\mathbb{R})$ is called \emph{$m$-subharmonic} with respect to $\omega$ if:
\begin{enumerate}
    \item $u$ is upper semicontinuous in $D$,
    \item $dd^c u \wedge dd^c u_2 \wedge \cdots \wedge dd^c u_m \wedge \omega^{n-m} \geq 0$ for all $u_2,\dots,u_m \in C^2(D)$ that are $m$-sh with respect to $\omega$,
    \item If $v \in L^1(D)$ satisfies the above conditions and $u = v$ a.e. in $D$, then $u \leq v$.
\end{enumerate}
\end{defi}

\begin{defi}
A function $u \in L^1(\Omega,\omega^n)$ is called \emph{$\omega$-$m$-subharmonic} ($\omega$-$m$-sh) if locally in any coordinate chart $D\subset \Omega$ where $\omega = dd^c\rho$, the function $u + \rho$ is $m$-subharmonic with respect to $\omega$. The set of all such functions is denoted by $\SH_m(\Omega,\omega)$.
\end{defi}

By following Bedford-Taylor’s seminal works, we define the 
Monge-Ampère operator for bounded $\omega$-$m$-sh functions as follows:
\begin{defi}
For bounded $\omega$-$m$-sh functions $u_1, \dots, u_m$, the \emph{complex Hessian operator} is defined recursively as:
$$
H_m(u_1, \dots, u_m) := (\omega + dd^c u_1) \wedge \cdots \wedge (\omega + dd^c u_m) \wedge \omega^{n-m}.
$$
This gives a positive Borel measure. When $u_1 = \cdots = u_m = u$, we write $H_m(u)$.
\end{defi}
\begin{rem}
   Classical results concerning bounded $\omega$-$m$-subharmonic functions $-$ such as the continuity of the Hessian measure with respect to decreasing sequences, the quasi-continuity of $\omega$-$m$-subharmonic functions, and properties of the $m$-capacity $-$ can be extended without difficulty to the setting considered here.

\end{rem}
\begin{defi} Let $\mathcal{D}_m(\Omega, \omega) $ be the set of functions $ \varphi \in \mathcal{SH}_m(\Omega, \omega) $ for which there exists a positive Radon measure $ \mu $ with the following property: if $ \{ \varphi_j \} $ is any sequence of bounded $\omega$-$m$-subharmonic functions decreasing to $ \varphi $, then
$
H_m(\varphi_j) \longrightarrow \mu
$
as $j\to +\infty$ in the weak sense of measures. We set
$
H_m(\varphi) := \mu.
$

\end{defi}

According to this definition, $\mathcal{D}_m(\Omega, \omega) $ is the largest set of $\omega$-$m$-subharmonic functions on which the Hessian operator $ (\omega + dd^c \cdot)^m\wedge\omega^{n-m} $ can be defined so that it is continuous with respect to decreasing sequences of bounded $\omega$-$m$-subharmonic functions.\\
\noindent

\begin{defi}
    We say that a domain $\Omega\subset X$ is quasi-$m$-hyperconvex if there exists a continuous $\omega$-$m$-subharmonic function $\rho : \Omega\to \mathbb{R}^-$ such that 
    $\{\rho<c\}\Subset\Omega$ for all $c<0$.
\end{defi}
\noindent Throughout this paper, we only consider 
quasi-$m$-hyperconvex domains $\Omega$ satisfying $\displaystyle\int_\Omega\omega^n<\int_X\omega^n.$ This condition will be exploited to show the existence of a subextension for functions in the class $\mathcal{F}_m(\Omega,\omega)$ (see Theorem \ref{ExtFm}).\\
\noindent The following proposition is a special case of \cite{CN}:
\begin{pro}\label{DP}
Let $(X, \omega)$ be a complex K\"ahler manifold of dimension $n$, and let $B \Subset X$ be a small relatively compact ball in a holomorphic chart. Let $1 \leq m \leq n$, and let $\varphi \in C^0(\partial B)$ be a continuous boundary data. Then there exists a function $u \in \mathcal{SH}_m(B, \omega) \cap C^0(\overline{B})$ such that
$$
(\omega + dd^c u)^m \wedge \omega^{n-m} = 0 \quad \text{in } B, \qquad u|_{\partial B} = \varphi.
$$
\end{pro}
\begin{thm}
\label{contperron}
Let $(X,\omega)$ be a compact K\"ahler manifold, $\Omega \subset X$ an $\omega$-$m$-pseudoconvex domain, and $f \in C^\infty(\overline{\Omega})$. The Perron envelope defined by
\[
P(f)(x) := \sup \{ v(x) \mid v \in \mathcal{SH}_{m}(\Omega,\omega),\ v \leq f \text{ in } \Omega\}
\]
satisfies:
\begin{enumerate}[label=(\roman*)]
    \item $P(f) \in C^0(\Omega)$
    \item $H_m(P(f)) = 0$ in $\{P(f) < f\}$.
\end{enumerate}
\end{thm}

\begin{proof}
Since $P(f)^\star\leq f$, then $P(f)=P(f)^\star$ and hence it is upper semicontinuous. Let $B \Subset \Omega$ be a small ball. By upper semicontinuity of $P(f)$, there exists a sequence $f_j \in C^\infty(\partial B)$ such that:
\[
f_j \searrow P(f)|_{\partial B} \quad \text{and} \quad f_j \leq f|_{\partial B}.
\]
Let $F$ be the smooth function such that $F\omega^n=H_m(f)$ on $\bar{B}$, $F^+=\max\{F,0\}$ and $H_m(f)^+(z):=\max\{(\omega +dd^cf)^m\wedge \omega^{n-m}(z),0\}$ is the non-negative part of $H_m(f)$. We choose a sequence of smooth non-negative functions $F_\ell$ decreasing uniformly to $F^+$ as $\ell\to+\infty$. 
For $\beta \geq 1$, consider the Dirichlet problem:
\begin{equation}\label{eq:penalized}
\begin{cases}
H_m(u_{\beta,j}) = e^{\beta({u_{\beta,j} - f})} \left( F_\ell + \frac{1}{\beta} \right)\omega^n & \text{in } B \\
u_{\beta,j} = f_j & \text{on } \partial B
\end{cases}
\end{equation}
By \cite[Lemma 3.16]{GN}, the Dirichlet problem \eqref{eq:penalized} admits a unique solution $u_{\beta,j,\ell}\in \mathcal{SH}_m(B,\omega)\cap\mathcal{C}^\infty(\overline{B})$. We claim that $u_{\beta,j} \leq f$. Indeed, suppose by contradiction that $\max_B (u_{\beta,j,\ell} - f) > 0$ and let $x_0\in \bar{B}$ be a point where this maximum is attained. Since $u_{\beta,j,\ell} = f_j \leq f$ on $\partial B$, then $x_0\in B$. Hence
\[
(u_{\beta,j,\ell} - f)(x_0) > 0 \implies H_m(u_{\beta,j,\ell})(x_0) \leq H_m(f)(x_0).
\]
On the other hand,
\[
H_m(u_{\beta,j})(x_0)=  e^{\beta({u_{\beta,j} - f})(x_0)} \left( F_\ell(x_0) + \frac{1}{\beta} \right)\omega^n>  [H_m(f)]^+(x_0)\geq H_m(f)(x_0),
\]
a contradiction. Thus, $u_{\beta,j,\ell} \leq f$ in $B$, so the family is uniformly bounded from above. By the same argument, since $f_{j_1} \geq f_{j_2}$ if $j_1<j_2$, one deduces
\[
u_{\beta,j_1,\ell} \geq u_{\beta,j_2,\ell} \quad \text{in}\;\; B.
\]
Thus $(u_{\beta,j,\ell})_j$ is a decreasing sequence. Let $u_{\beta,\ell}=\lim_ju_{\beta,j,\ell}$. We have $u_{\beta,\ell}=P(f)$ on $\partial B$. 
Let $\rho\in\mathcal{C}^\infty(\overline{B})$ such that $dd^c\rho=\omega$ in $B$ and $\rho=0$ on $\partial B$. For $A$ large enough, we have $H_m(P(f)+A\rho)\geq H_m(u_{\beta,\ell})$. Since $u_{\beta,\ell}=P(f)+A\rho$ on $\partial B$, by the comparison principle (see \cite[Corollary 3.11]{GN}) we have $u_{\beta,\ell}\geq P(f)+A\rho$. Then we have a uniform lower bound of $u_{\beta,j,\ell}$ independent of $j$, $\ell$ and $\beta$.
Therefore, the family $\{u_{\beta,j,\ell}\}$ satisfies uniform $L^\infty$ bounds. From \eqref{eq:penalized} we write the right-hand side as
\[
g_{\beta,j,\ell} := e^{\beta(u_{\beta,j,\ell}-f)}\Big(F_\ell+\tfrac{1}{\beta}\Big).
\]
Since $u_{\beta,j,\ell}\le f$, we have $e^{\beta(u_{\beta,j,\ell}-f)}\le 1$, hence
$
0\le g_{\beta,j,\ell}\le M+1,
$
independently of $\beta,j, \ell$. Thus, the measures $H_m(u_{\beta,j,\ell})=g_{\beta,j,\ell}\,\omega^n$ have densities uniformly bounded from above. Hence $g_{\beta,j,\ell}\in L^q(B)$ for all $q$ and $\|g_{\beta,j,\ell}\|_{L^q}$ is bounded uniformly. Then by \cite[Theorem 2.7]{DK}, the family $\{u_{\beta,j,\ell}\}_{j}$ is equicontinuous in each compact $K\Subset B$. From the Arzelà-Ascoli theorem it follows that $u_{\beta,\ell}=\lim_ju_{\beta,j,\ell}$ is continuous on $B$. 
Using the same arguments, we show that the sequence $(u_{\beta,\ell})_\ell$ is increasing in $B$. Applying \cite[Theorem 2.5]{DK} on every ball $B^\prime\Subset B$ we conclude that $u_{\beta,\ell}$ converges uniformly on every compact subset of $B$ to a function $u_\beta$. Then $u_\beta=P(f)$ on $\partial B$, continuous on $B$ and $H_m(u_\beta)=e^{\beta(u_\beta-f)}(H_m(f)^++\frac{\omega}{\beta})$. Also, by the same arguments we show that the sequence $(u_{\beta})_\beta$ is increasing in $B$ and converges uniformly on every compact subset of $B$ to a function $u$. The function $u$ satisfies: $u=P(f)$ in $\partial B$, continuous in $B$ and $H_m(u)={\bf 1}_{\{u=f\}}H_m(f)^+$.
Let $G := \{ z \in B : u(z) < f(z) \}$. Since the set $G\subset B$ is open, then $u = f\geq P(f)$ on $\partial G$ and $(\omega + dd^c u)^m \wedge \omega^{n-m} = 0 \quad \text{in } G.$  It follows from \cite[Corollary 3.11]{GN} that $u\geq P(f)$. Conversely, define:
\[
v = \begin{cases}
u & \text{in } B \\
P(f) & \text{in } \Omega \setminus B
\end{cases}
\]
Then $v \in \mathcal{SH}_m{(\Omega)}$ and $v \leq f$, so $v \leq P(f)$. In particular $u \leq P(f)$ on $B$.
Therefore, $u = P(f)$ in $B$. Since $B$ is arbitrary in $\Omega$, we conclude that $P(f)$ is continuous on $\Omega$ and that $H_m(P(f)) = 0$ on $\{P(f) < f\}$.
\end{proof}
\begin{pro}
\label{approx-local}
Let $u$ be an $\omega$-$m$-subharmonic on $\Omega$. For every compact set $K\subset \Omega$, there exists an open set $V$ with $K\Subset V\Subset \Omega$ and a sequence of  functions $\{\psi_k\}_{k\ge1}\subset\mathcal{SH}_{m}(V,\omega)\cap \mathcal C^\infty(V)$  such that
\[
\psi_k \searrow u \quad\text{pointwise on }K.
\]
\end{pro}

\begin{proof} {\bf Step 1.} Suppose first that $u$ is continuous and strictly $\omega$-$m$-subharmonic on $\Omega$. We repeat the same proof given in $\cite{Plis}$.
 Cover $K$ by a finite collection of open sets $(U_j)_{1\le j\le N}$ contained in $\Omega$ such that each $U_j$ is a smooth strictly pseudoconvex domain. Let $U_j'\Subset U_j$ be open sets such that $W=\bigcup_j U_j'\supset K$.
By \cite[Lemma 9]{Plis}, for each $j$ and for $k\geq 1$, there exists a function $v_{j,k}\in\mathcal{C}^\infty(\overline{U_j})$ strictly $\omega$-$m$-subharmonic in $U_j$ such that
\[
v_{j,k} > u \;\text{on }\overline{U_j'},\; and\; v_{j,k} \le u + \delta_{j,k} \;\text{on }\partial U_j,
\]
where $\delta_{j,k}>0$ is chosen arbitrarily small. For each \(j,k\), we can modify \(v_{j,k}\) outside \(\overline{U_j'}\) and extend it to a function \(u_{j,k}\in\mathcal{C}^\infty(\Omega)\) such that
\(
u_{j,k} > u\) on \(\overline{U_j'}\),
\(
u_{j,k} < u + \delta_{j,k}\) on \(\Omega\),
\(
u_{j,k} < u\) on \(\Omega\setminus U_j,
\)
and \(u_{j,k}\) is strictly \(\omega\)-\(m\)-subharmonic on the set
\(\displaystyle
\bigl\{u_{j,k}-u > \tfrac12 \inf_{U_j}(u_{j,k}-u)\bigr\}.
\)
For $t>0$, define the smooth function on $W$:
\[
\Psi_{k,t} := \frac{1}{t}\log\left(\sum_{j=1}^N e^{t u_{j,k}}\right).
\]
 For sufficiently large $t$, $\Psi_{k,t}$ is strictly $\omega$-$m$-subharmonic on a neighborhood $V$ of $K$, and we have the estimate
\(\displaystyle
\max_j u_{j,k} \le \Psi_{k,t} \le \max_j u_{j,k} + \frac{\log N}{t}.
\)

By choosing appropriate $\delta_{j,k}$ and $t$, we obtain a smooth function $\psi_k := \Psi_{k,t}$ strictly $\omega$-$m$-subharmonic on $V$ satisfying $
u < \psi_k \le u + \varepsilon_k$ on $K$,
with $\varepsilon_k > 0$ as small as desired.
Let $\varepsilon_1 > 0$ be small. For $k\geq 1$, define recursively
\(\displaystyle
\varepsilon_{k+1} := \frac{1}{2}\min_{K}(\psi_k - u) > 0.
\)
Applying the previous construction with a sufficiently small $\delta_{j,k+1}$, we obtain $\psi_{k+1}$ such that
\(
u \le \psi_{k+1} \le u + \varepsilon_{k+1} \) on \(K\).
By construction, $\varepsilon_{k+1} \leq \dfrac{1}{2}(\psi_k - u)$. Hence,
\(\displaystyle
\psi_{k+1} \le \frac{u + \psi_k}{2} \le \psi_k \) on \(K\).
The sequence $\{\psi_k\}$ is decreasing and converges uniformly to $u$ on $K$ since $\|\psi_k - u\|_{L^\infty(K)} \leq \varepsilon_k \to 0$.\\
{\bf Step 2.} Now take $u\in \SH_m(\Omega,\omega)$. Since $u$ is $\omega$-$m$-subharmonic and hence upper semicontinuous, there exists a sequence $\{f_j\} \subset \mathcal{C}^\infty(\overline{\Omega})$ such that
\[
f_j \searrow u \quad \text{on } \Omega.
\]
By Theorem \ref{contperron}, $P(f_j)$ is continuous and $\omega$-$m$-subharmonic. Moreover, since $u \in \SH_m(\Omega,\omega)$ and $u \leq f_j$, we have
\(
u \leq P(f_j) \leq f_j.
\)
Therefore, $P(f_j) \searrow u$ pointwise on $\Omega$.
 Since we are interested in proving the approximation in a neighborhood of the compact set $K$, we may assume that $u$ and $P(f_j)$ are negative. Let $0\leq \chi \in \C^\infty(\bar{\Omega})$ be a strictly $\omega$-$m$-subharmonic function on $\Omega$ and $\{\delta_j\}$ be a sequence of positive numbers decreasing to $0$. Set \[
v_j := (1-\delta_j) P(f_j) + \delta_j \chi.
\]
Then the function $v_j$ is continuous and  strictly $\omega$-$m$-subharmonic on \(\Omega\). Moreover, since $P(f_j) \searrow u$ and $\delta_j \searrow 0$, we have $v_j \searrow u$ pointwise on $\Omega$. 
 Choose an open set $V$ such that $K \Subset V \Subset \Omega$. By Step 1, for each $j$, there exists a sequence $\{\psi_{j,k}\}_{k \geq 1} \subset \mathcal{C}^\infty(\bar{V})$ of strictly $\omega$-$m$-subharmonic functions in a neighborhood of $\bar{V}$ such that
\(
\psi_{j,k} \searrow v_j \,\text{uniformly on } \bar{V}.
\)
In particular, for each $j$, there exists an index $N(j)$ such that
\[
v_j + \frac{1}{j+1} \leq \psi_{j,N(j)} \leq v_j + \frac{1}{j} \quad \text{on } \bar{V}.
\]
Define $\psi_j = \psi_{j,N(j)}$.
We now verify that the sequence $\{\psi_j\}$ is decreasing and converges pointwise to $u$ on $K$. First, for each $j$, we have
\(
\psi_j \geq v_j + \frac{1}{j+1} \geq v_{j+1} + \frac{1}{j+1}\)
and
\(
\psi_{j+1} \leq v_{j+1} + \frac{1}{j+1}.
\)
Thus, $\psi_j \geq \psi_{j+1}$ on $K$, so the sequence $\{\psi_j\}$ is decreasing.
Moreover, for each $x \in K$, we have:
\[
u(x) \leq \psi_j(x) \leq v_j(x) + \frac{1}{j}.
\]
Since $v_j(x) \searrow u(x)$, it follows that $\psi_j(x) \to u(x)$. Therefore, $\psi_j \searrow u$ pointwise on $K$.
This completes the proof of the proposition.

\end{proof}

\noindent The following lemma can be obtained by classical methods involving capacity and quasi-continuity. For the reader's convenience, we provide a complete proof below.
\begin{lem}\label{lemconv}
Let $(\mu_j)$ be a sequence of measures with uniformly bounded mass converging weakly to $\mu$. Suppose that for all $\varepsilon>0$ there exists $\delta>0$ such that for every $E\subset \Omega$ with $cap_{\omega,m}(E)<\delta$ we have $\forall j\geq 0$, $\mu_j(E)<\varepsilon$. Then for all $\varphi, \psi\in \mathcal{SH}_m(\Omega,\omega)$ we have  $$ \int_{\{\psi<\varphi\}}d\mu\leq \liminf_j\int_{\{\psi<\varphi\}}d\mu_j.$$ 
\end{lem}
\begin{proof}
Fix $\varepsilon > 0$. By quasi-continuity of $\omega$-$m$-subharmonic functions, there exist open sets $U_\varphi, U_\psi \subset \Omega$ with 
\[
\mathrm{cap}_{\omega,m}(U_\varphi) < \delta/2 \quad \text{and} \quad \mathrm{cap}_{\omega,m}(U_\psi) < \delta/2,
\]
such that $\varphi$ is continuous on $\Omega \setminus U_\varphi$ and $\psi$ is continuous on $\Omega\setminus U_\psi$. Let $U := U_\varphi \cup U_\psi$, so that $\mathrm{cap}_{\omega,m}(U) < \delta$. By assumption, it follows that $\mu_j(U) < \varepsilon$ for all $j$.
Let $K$ be any compact subset of $\{\psi < \varphi\}$. Then
\[
\mu(K) = \mu(K \setminus U) + \mu(K \cap U) \leq \mu(K \setminus U) + \mu(U).
\]
Since $U$ is open and $\mu_j(U) < \varepsilon$ for all $j$, weak convergence implies
\(
\mu(U) \leq \liminf_{j \to \infty} \mu_j(U) \leq \varepsilon.
\)
Let $L := K \setminus U$, and let $\tilde{\varphi}$ be continuous on $\Omega$ such that $\tilde{\varphi}=\varphi$ on $\Omega\setminus U$. Then $L$ is compact (as the intersection of a compact set and a closed set) and $L \subseteq \{\psi < \varphi\} \cap (\Omega \setminus U)=\{\psi < \tilde{\varphi}\} \cap (\Omega \setminus U)$. Since the set $\{\psi < \tilde{\varphi}\}$ is open in $\Omega$, there exists an open set $V \subset \Omega$ such that
\[
L \subset V \quad \text{and} \quad \overline{V} \Subset \{\psi < \tilde{\varphi}\},
\]
where $\overline{V}$ denotes the closure of $V$ in $\Omega$. Let $\phi : \Omega \to [0,1]$ be a continuous function such that $\phi \equiv 1$ on $L$ and $\mathrm{supp}(\phi) \subset V$. Such a function exists by the Urysohn lemma. For every $j$, we have
\(\int_{\{\psi < \varphi\}\cup U} d\mu_j\geq  
\int_{\{\psi < \tilde{\varphi}\}} d\mu_j \geq \int \phi \, d\mu_j,
\)
since $\mathrm{supp}(\phi) \subset V \subset \{\psi < \tilde\varphi\}$ and \(\{\psi < \tilde\varphi\}\subset \{\psi < \varphi\}\cup U\). By weak convergence,
\[
\liminf_{j \to \infty} \int \phi \, d\mu_j = \int \phi \, d\mu.
\]
Moreover,
\(
\int \phi \, d\mu \geq \mu(L),
\)
because $\phi \equiv 1$ on $L$. Thus,
\(
\liminf_{j \to \infty} \int_{\{\psi < \varphi\}} d\mu_j +\varepsilon\geq \mu(L).
\)
Combining the estimates:
\[
\mu(K) \leq \mu(L) + \mu(U) \leq \liminf_{j \to \infty} \int_{\{\psi < \varphi\}} d\mu_j + 2\varepsilon.
\]
Since $\varepsilon > 0$ is arbitrary, then  
\(
\mu(K) \leq \liminf_{j \to \infty} \int_{\{\psi < \varphi\}} d\mu_j
\)
for all such compact subsets $K$ of  $\{\psi < \varphi\}$. By inner regularity of $\mu$,
\[
\mu(\{\psi < \varphi\}) = \sup \left\{ \mu(K) : K \text{ compact} \subset \{\psi < \varphi\} \right\} \leq \liminf_{j \to \infty} \int_{\{\psi < \varphi\}} d\mu_j,
\]
which completes the proof.

\end{proof}
\subsection{Comparison principle}
\noindent The following proposition can be obtained as in the local theory (see \cite{BT2}, \cite{BT3}).
\begin{pro}\label{demailly}
Let $T$ be a closed $(\omega,m)$-positive
current of the form $T=(dd^cu_1+\omega)\wedge\dots\wedge(dd^cu_k+\omega)\wedge \omega^{n-k}$ with $1\leq k\leq m$ and $\varphi, \psi \in \mathcal{SH}_m(\Omega, \omega) \cap L^\infty(\Omega) $. Then  
$$
1_{\{\varphi < \psi\}} (\omega + dd^c \max\{\varphi, \psi\})^l \land T = 1_{\{\varphi < \psi\}} (\omega + dd^c \psi)^l \land T, 
$$
in the weak sense of Borel measures on $ \Omega $, with $l+k=m$. In particular 
$$
1_{\{\varphi \leq \psi\}} (\omega + dd^c \max\{\varphi, \psi\})^l \land T \geq 1_{\{\varphi \leq \psi\}} (\omega + dd^c \psi)^l \land T, 
$$
in the weak sense of Borel measures on $ \Omega $.
\end{pro}
\begin{lem}\label{PC}
    Let $T$ be a closed $(\omega,m)$ positive
current of the form $T=(dd^cu_1+\omega)\wedge\dots\wedge(dd^cu_k+\omega)\wedge \omega^{n-m}$ with $1\leq k\leq m$ and let $\varphi, \psi \in\mathcal{SH}_m(\Omega,\omega)\cap L^\infty(\Omega)$ be such that $\lim_{z\to\partial\Omega}(\varphi-\psi)_*(z)\geq 0$. Assume that $\int_\Omega(dd^c\varphi+\omega)^p\wedge T<+\infty
$ with $p+k=m$. Then $$
\int_{\{\varphi<\psi\}}(dd^c\psi+\omega)^p\wedge T\leq \int_{\{\varphi<\psi\}}(dd^c\varphi+\omega)^p\wedge T\quad\mbox{and}\; \int_{\{\varphi\leq\psi\}}(dd^c\psi+\omega)^p\wedge T\leq \int_{\{\varphi\leq\psi\}}(dd^c\varphi+\omega)^p\wedge T.
$$
\end{lem}

\begin{proof}
Observe that the condition $(\varphi - \psi)_* \geq 0$ implies that for every $\varepsilon > 0$, the set $\{\varphi < \psi - \varepsilon\}$ is relatively compact in $\Omega$. Replacing $\psi$ with $\psi - \varepsilon$ and letting $\varepsilon \searrow 0$, we may assume without loss of generality that $\{\varphi < \psi\} \Subset \Omega$.

Define $\tilde{\varphi} := \sup\{\varphi, \psi\}$. Then $\tilde{\varphi} \in \mathcal{SH}_m(\Omega, \omega) \cap L^\infty(\Omega)$ and coincides with $\varphi$ in a neighborhood of the boundary of $\Omega$. Applying Proposition \ref{approx-local} together with  Stokes' formula, yields:
$$
\int_\Omega (\omega + dd^c \tilde{\varphi})^p \wedge T = \int_\Omega (\omega + dd^c \varphi)^p \wedge T.
$$
Moreover, by Proposition \ref{demailly}, we have
$
\int_{\{\varphi < \psi\}} (\omega+dd^c\psi)^p \wedge T = \int_{\{\varphi < \psi\}} (\omega + dd^c \tilde{\varphi})^p \wedge T.
$
Consequently,
\begin{align*}
\int_{\{\varphi < \psi\}} (\omega+dd^c\psi)^p \wedge T &= \int_\Omega (\omega + dd^c \tilde{\varphi})^p \wedge T - \int_{\{\varphi \geq \psi\}} (\omega + dd^c \tilde{\varphi})^p \wedge T \\
&\leq \int_\Omega (\omega+dd^c\varphi)^p \wedge T - \int_{\{\varphi > \psi\}} (\omega + dd^c \tilde{\varphi})^p \wedge T \\
&= \int_\Omega (\omega+dd^c\varphi)^p \wedge T - \int_{\{\varphi > \psi\}} (\omega+dd^c\varphi)^p \wedge T,
\end{align*}
which leads to the inequality
$$
\int_{\{\varphi < \psi\}} (\omega+dd^c\psi)^p \wedge T \leq \int_{\{\varphi \leq \psi\}} (\omega+dd^c\varphi)^p \wedge T.
$$
Applying the above inequality to $\varphi + \varepsilon$ and $\psi$, and passing to the limit as $\varepsilon \to 0$, we obtain the desired result.
To prove the second inequality, we may assume that $\varphi, \psi < 0$ in $\Omega$. Consider the previous inequality applied to $\varphi$ and $t\psi$, with $0 < t < 1$. Note that
$$
(dd^c(t\psi) + \omega)^p \geq t^p (\omega+dd^c\psi)^p.
$$
Letting $t \to 1$ yields the claimed inequality.
\end{proof}

 \begin{cor}\label{compar}
Let $T$ be a closed $(\omega,m)$-positive
current of the form $T=(dd^cu_1+\omega)\wedge\dots\wedge(dd^cu_k+\omega)\wedge \omega^{n-k}$ with $1\leq k\leq m$ and let $\varphi, \psi \in\mathcal{SH}_m(\Omega,\omega)\cap L^\infty(\Omega)$ such that $\varphi\leq \psi$. Assume that $\int_\Omega(dd^c\varphi+\omega)^l\wedge T<+\infty
$ with $l+k=m$. Then $$
\int_\Omega(dd^c\psi+\omega)^l\wedge T\leq \int_\Omega(dd^c\varphi+\omega)^l\wedge T.
$$
 \end{cor}   

\begin{thm}\label{pcf}
Let $X$ be a compact K\"ahler manifold of dimension $n$ with K\"ahler form $\omega > 0$, and let $\Omega \subset X$ be a domain.
Let $u, v \in \mathcal{SH}_m(\Omega, \omega)$ be $\omega$-\textit{m}-subharmonic functions such that:
\begin{enumerate}
    \item $\displaystyle \liminf_{z \to \partial \Omega} (u(z) - v(z)) \geq 0$,
    \item $(\omega + dd^c u)^m \wedge \omega^{n - m} \leq (\omega + dd^c v)^m \wedge \omega^{n - m}$ as measures on $\Omega$.
\end{enumerate}
Then $u \geq v$ on $\Omega$.
\end{thm}
\begin{proof}
Assume by contradiction that $v > u$ somewhere in $\Omega$. Let $K \Subset \Omega$ be a  compact set  with $v - u \geq \delta > 0$ on $K$ and let $\chi \in C_c^\infty(\Omega)$ such that:
\begin{itemize}
    \item $0 \leq \chi \leq 1$,
    \item $\chi = 1$ on a neighborhood of $K$,
    \item $\omega + dd^c \chi \geq \delta_0 \omega$ for some $\delta_0 > 0$.
\end{itemize}
For $\varepsilon > 0$, define
$
\varphi_\varepsilon := (1 - \varepsilon) v + \varepsilon \chi$. We have $\varphi_\varepsilon \in \mathcal{SH}_m(\Omega, \omega)$.
Let
$
E_\varepsilon := \{ z \in \Omega : (1-\varepsilon)u(z) < \varphi_\varepsilon(z) \}.
$
For a small $\varepsilon$, we have $K \subset E_\varepsilon \subset \Omega$.
Then by Lemma \ref{PC} we have:
$$
\int_{E_\varepsilon} (1-\varepsilon)^m(\omega + dd^c u)^m \wedge \omega^{n - m}
\geq \int_{E_\varepsilon} (\omega + dd^c \varphi_\varepsilon)^m \wedge \omega^{n - m}.
$$
Since
$
(\omega + dd^c \varphi_\varepsilon)^m\wedge\omega^{n-m} 
\geq (1 - \varepsilon)^m (\omega + dd^c v)^m\wedge\omega^{n-m} + \varepsilon^m (\omega + dd^c \chi)^m\wedge\omega^{n-m}
$, then:
$$
\int_{E_\varepsilon} (\omega + dd^c u)^m \wedge \omega^{n - m}
\geq \int_{E_\varepsilon} (\omega + dd^c v)^m \wedge \omega^{n - m}  + \frac{\varepsilon^m}{(1-\varepsilon)^m} \int_{E_\varepsilon} (\omega + dd^c \chi)^m \wedge \omega^{n - m}.
$$
On the other hand, it follows from the hypothesis $
(\omega + dd^c u)^m \wedge \omega^{n - m} \leq (\omega + dd^c v)^m \wedge \omega^{n - m},$
that:
$$
 \int_{E_\varepsilon} (\omega + dd^c u)^m \wedge \omega^{n - m}
\leq \int_{E_\varepsilon} (\omega + dd^c v)^m \wedge \omega^{n - m}
$$
which contradicts the previous inequality, since the last term is strictly positive.
Therefore, $u \geq v$ in $\Omega$.
\end{proof}
\subsection{The class $\mathcal{E}^0_m(\Omega,\omega)$}
\begin{defi} Let $\Omega$ be a quasi-$m$-hyperconvex domain. We denote by $\mathcal{E}^0_m(\Omega,\omega)$ the class of bounded functions $\varphi\in\mathcal{SH}_m(\Omega,\omega)\cap L^\infty(\Omega)$ such that $\lim_{z\to\partial\Omega}\varphi(z)=0$ and $\int_\Omega H_m(\varphi)<+\infty.$
\end{defi}
\begin{pro}\label{cone}
\begin{enumerate}
    \item The class $\mathcal{E}_m^0(\Omega,\omega)$ is convex and satisfies the lattice property: If $\varphi \in \mathcal{E}_m^0(\Omega,\omega)$ and $ u \in \mathcal{SH}_m^{-}(\Omega,\omega)$, then $\sup\{\varphi,u\} \in \mathcal{E}_m^0(\Omega,\omega)$.
\item Let $1 \leq p, q$ be integers such that $p+q \leq m$, and $T=(dd^cu_s+\omega)\wedge\dots\wedge(dd^cu_m+\omega)\wedge \omega^{n-m}$ with $s := m - p - q$. Then for all $\varphi,\psi \in \mathcal{E}_m^0(\Omega,\omega)$, we have
$$
\int_\Omega (\omega + dd^c\varphi )^p \wedge (\omega + dd^c\psi)^q \wedge T \leq \int_\Omega (\omega + dd^c\varphi )^{p+q} \wedge T + \int_\Omega (\omega + dd^c\psi)^{p+q} \wedge T. 
$$
\item If $\varphi_1,\dots,\varphi_m \in \mathcal{E}_m^0(\Omega,\omega)$, then
$$\displaystyle
\int_\Omega (\omega + dd^c\varphi_1)\wedge \cdots \wedge (\omega+dd^c\varphi_m)\wedge \omega^{m-n} \leq 2^{m-1} \sum_{j=1}^{m} \int_\Omega (\omega+dd^c\varphi_j)^m\wedge\omega^{n-m}.
$$
\end{enumerate}
\end{pro}
\begin{proof}
The proposition follows by applying Lemma \ref{PC} and reproducing the proof of \cite[Corollary 3.5]{CKZ2} step by step.
\end{proof}
\begin{pro}\label{convs}
Let $1 \leq m \leq n$ and $\Omega \subset X$ be a quasi-$m$-hyperconvex domain in a compact K\"ahler manifold $(X, \omega)$. Let
$
\varphi_j^0, \dots, \varphi_j^m \in \mathcal{E}^0_m(\Omega,\omega)$, converging monotonically to $
\varphi^0, \dots, \varphi^m \in \mathcal{E}^0_m(\Omega,\omega)$.
Then
$$
\lim_{j \to \infty} \int_\Omega (-\varphi_j^0)\, (\omega+dd^c \varphi_j^1) \wedge \cdots \wedge (\omega+dd^c \varphi_j^m) \wedge \omega^{n - m}
= \int_\Omega (-\varphi^0)\, (\omega+dd^c \varphi^1) \wedge \cdots \wedge (\omega+dd^c \varphi^m) \wedge \omega^{n - m}.
$$
\end{pro}
\begin{proof}
Let us denote
$
H_j := (\omega+dd^c \varphi_j^1) \wedge \cdots \wedge (\omega+dd^c \varphi_j^m) \wedge \omega^{n - m}$ and $ H := (\omega+dd^c \varphi^1) \wedge \cdots \wedge (\omega+dd^c \varphi^m) \wedge \omega^{n - m}.
$ By Proposition \ref{cone}, the currents $H_j$ and $H$ have uniformly bounded total mass on $\Omega$. We prove first that $\mu_j:=-\varphi_j^0 H_j$ converges weakly to $\mu_0:=-\varphi^0 H$. Since all functions are bounded and $
\varphi_j^0, \dots, \varphi_j^m $, converge monotonically to $
\varphi^0, \dots, \varphi^m \in \mathcal{E}^0_m(\Omega,\omega)$, then there exists $M\geq1$ such that  the functions $ \varphi_j^k $ and $ \varphi^k $ take values in the interval $[-M, 0]$. Let $ \chi \in C_c(\Omega) $ be a test function with support contained in $ K\Subset\Omega $. For any fixed $ \varepsilon > 0 $, there exists an open set $ G \subset \Omega $ such that
$$
\operatorname{cap}_{\omega,m}(G) < \varepsilon
\quad \text{and} \quad
\varphi^0_j, \varphi^0 \in C(K \setminus G).
$$
The monotonicity of the sequence $ (\varphi_j^0) $ implies that $ \varphi_j^0 \to \varphi^0 $ uniformly on $ K \setminus G $. Put $u_j=\frac{1}{mM}\sum_{k=1}^m\varphi_j^k$. Then $(\omega+dd^cu_j)^m\wedge\omega^{n-m}\geq (\frac{1}{mM})^m\,H_j$ and $-1\leq u_j<0$. Therefore, we obtain the following estimate:
$$\begin{array}{ll}
    \displaystyle\int_K |(\varphi_j^0 - \varphi^0)\chi|\,H_j
&\leq \displaystyle (mM)^m\| (\varphi_j^0 - \varphi^0)\chi \|_{K \setminus G}
\int_K (\omega+dd^c u_j)^m \wedge \omega^{n - m}\\
&\qquad + \displaystyle(mM)^m\| (\varphi_j^0 - \varphi^0)\chi \|_K
\int_G (\omega+dd^c u_j)^m \wedge \omega^{n - m}\\
&\displaystyle\leq (mM)^m\| (\varphi_j^0 - \varphi^0)\chi \|_{K \setminus G} \cdot \operatorname{cap}_{\omega,m}(K)
+ (mM)^m\| (\varphi_j^0 - \varphi^0)\chi \|_K \cdot \operatorname{cap}_{\omega,m}(G).
\end{array}
$$
Taking the upper limit as $ j \to +\infty $, and then letting $ \varepsilon \to 0 $, yields $
\int_K |(\varphi_j^0 - \varphi^0)\chi|\, (dd^c u_j)^m \wedge \omega^{n - m}
\longrightarrow 0. 
$ Since $-\varphi^0 H_j$ converges weakly to $-\varphi^0 H$, we get the desired result.\\ Now, let $\varepsilon>0$ and $D\Subset\Omega$ such that $-\varepsilon\leq \varphi_j^0<0$, $-\varepsilon\leq \varphi^0<0$ on $\Omega\setminus D$ and $\mu_0$ puts no mass on $\partial{D}$. Then 
$$
\int_\Omega-\varphi_j^0 \, H_j-\int_\Omega-\varphi\, H=\int_D-\varphi_j^0 \, H_j-\int_D-\varphi\, H\, + O(\varepsilon).
$$
It follows from the first step, that $
\mu_0(D)\leq \liminf_j\mu_j(D)\leq\limsup_j\mu_j(\bar{D})\leq\mu_0(\bar{D})=\mu_0(D).
$
Hence, taking the upper limit as $ j \to +\infty $, and then letting $ \varepsilon \to 0 $, yields $
\int_\Omega-\varphi_j^0 \, H_j-\int_\Omega-\varphi\, H
\longrightarrow 0.$ 
\end{proof}

\begin{thm}\label{IBP}
  Let $T$ be a closed $(\omega,m)$-positive
current of the form $T=(dd^cu_2+\omega)\wedge\dots\wedge(dd^cu_m+\omega)\wedge \omega^{n-m}$, where $u_2, \dots, u_m\in \mathcal{E}^0_m(\Omega,\omega)$, and let $\varphi, \psi \in\mathcal{E}^0_m(\Omega,\omega)$. Then    
$\displaystyle
\int_\Omega \psi dd^c\varphi\wedge T=\int_\Omega \varphi dd^c\psi\wedge T $
 and $$
\int_\Omega \psi (dd^c\varphi+\omega)\wedge T-\int_\Omega \varphi (dd^c\psi+\omega)\wedge T= \int_\Omega (\psi-\varphi)\omega\wedge T.$$
\end{thm}
\begin{proof}
Let $ \varepsilon > 0 $ be small enough, set $ u_\varepsilon := \sup\{u, v - \varepsilon\} $ and $ v_\varepsilon := \sup\{v, u - \varepsilon\}.$  Observe that $ u_\varepsilon = u $ and $ v_\varepsilon = v $ near $ \partial \Omega $.  Applying Proposition \ref{approx-local} together with  Stokes' formula, we get 
$$
\int_\Omega u_\varepsilon\,dd^c v_\varepsilon \land T - v_\varepsilon\,dd^c u_\varepsilon \land T= \int_\Omega u\,dd^c v \land T - v\,dd^c u \land T. 
$$
Put $g := \max\{u, v\} $. Since $ u_\varepsilon \nearrow g $ and $ v_\varepsilon \nearrow g $, it follows from Proposition \ref{demailly} that
$$
\lim_{\varepsilon \to 0} \int_\Omega u_\varepsilon\,dd^c v_\varepsilon \land T = \int_\Omega g\,dd^c g \land T = \lim_{\varepsilon \to 0} \int_\Omega v_\varepsilon\,dd^c u_\varepsilon \land T.
$$
This yields the desired integration by parts formula.
   
\end{proof}
\begin{pro}\label{ccv}
Let $(X,\omega)$ be a compact K\"ahler manifold and $\Omega \subset X$ be a quasi-$m$-hyperconvex domain.  
Let $u \in \mathcal{E}_0^m(\Omega,\omega)$. 
Then there exists a sequence $({u}_j) \subset \mathcal{E}_0^m(\Omega,\omega) \cap C(\overline{\Omega})$ such that ${u}_j \searrow u$ pointwise on $\Omega$.
\end{pro}

\begin{proof}
Since $u$ is upper semicontinuous, there exists a decreasing sequence $(f_j)$ of smooth functions on $\bar{\Omega}$ such that $f_j \searrow u$ on $\Omega$ and $f_j\leq 0$. Set $u_j := P(f_j)$. We have $u_j \leq 0$, $u_j \searrow u$ on $\Omega$ and by Theorem \ref{contperron}, $u_j \in \mathcal{SH}_m(\Omega,\omega)\cap C(\overline{\Omega})$. Moreover, since $u\in\mathcal{E}_0^m(\Omega,\omega)$,
it follows from Corollary \ref{compar} that $u_j\in \mathcal{E}_0^m(\Omega,\omega) \cap C(\overline{\Omega})$.
\end{proof}
\subsection{Weighted Hessian energy classes for $ \omega $-$ m $-subharmonic functions}

\begin{defi}
Let $ \Omega \subset X $ be an open subset of a compact K\"ahler manifold $ (X, \omega) $. We say that $ \varphi \in \mathcal{F}_m(\Omega ,\omega) $ if there exists a decreasing sequence $ (\varphi_j) \subset \mathcal{E}_0^m(\Omega ,\omega) $ such that $ \varphi_j \searrow \varphi $ on $ \Omega  $ and
$$
\sup_j \int_\Omega  H_m(\varphi_j) < +\infty,
$$
where $ H_m(\varphi_j) := (\omega + dd^c \varphi_j)^m \wedge \omega^{n-m} $ denotes the complex Hessian measure.
\end{defi}

\begin{lem}\label{fcomp}
Let $ \varphi \in \mathcal{F}_m(\Omega ,\omega) $. Then the quantity
$$
\mathrm{H}_{\Omega}^m(\varphi):= \lim_j \int_\Omega  H_m(\varphi_j) = \sup_j \int_\Omega  H_m(\varphi_j)
$$
is independent of the approximating sequence $ (\varphi_j) \subset \mathcal{E}_0^m(\Omega ,\omega) $ decreasing to $ \varphi $.
Moreover, if $ \psi \in \mathcal{SH}_m(\Omega ,\omega) $ with $ \varphi \leq \psi \leq 0 $, then $ \psi \in \mathcal{F}_m(\Omega ,\omega) $ and $
\mathrm{H}_{\Omega}^m(\psi)\leq\mathrm{H}_{\Omega}^m(\varphi)
$.
\end{lem}
\begin{proof}
We prove first the independence of the approximating sequence.
Let $(\varphi_j)$ be a sequence in $\mathcal{E}^0_m(\Omega, \omega)$ decreasing to $\varphi$ with 
$
\sup_j \int_\Omega H_m(\varphi_j) < +\infty.
$
By the comparison principle, since $\varphi_{j+1} \leq \varphi_j$, we have
$$
\int_\Omega H_m(\varphi_j) \leq \int_\Omega H_m(\varphi_{j+1}) \quad \forall j.
$$
Thus the sequence $\left( \int_\Omega H_m(\varphi_j) \right)_j$ is 
increasing and bounded, hence converges to
$
L := \lim_j \int_\Omega H_m(\varphi_j) = \sup_j \int_\Omega H_m(\varphi_j).
$
Let $(\psi_k)$ be another sequence in $\mathcal{E}^0_m(\Omega, \omega)$ decreasing to $\varphi$. For any $\epsilon > 0$ and $k \in \mathbb{N}$, the continuity of the Hessian operator for bounded $\omega$-$m$-subharmonic functions implies that 
$
H_m(\sup\{\psi_k, \varphi_j\}) \to H_m(\psi_k)$ weakly on $\Omega$ as $j \to +\infty$.
Thus there exists $j_k$ such that
$
\int_\Omega H_m(\sup\{\psi_k, \varphi_{j_k}\}) > \int_\Omega H_m(\psi_k) - \epsilon.
$
Since $\sup\{\psi_k, \varphi_{j_k}\} \geq \varphi_{j_k}$, the comparison principle gives
$
\int_\Omega H_m(\sup\{\psi_k, \varphi_{j_k}\}) \leq \int_\Omega H_m(\varphi_{j_k}) \leq L.
$
Therefore,
$$
\int_\Omega H_m(\psi_k) - \epsilon \leq L.
$$
As this holds for all $\epsilon > 0$, we get $\sup_k \int_\Omega H_m(\psi_k) \leq L$. By symmetry (exchanging $(\varphi_j)$ and $(\psi_k)$), we obtain equality $
\lim_k \int_\Omega H_m(\psi_k) = L.$
Thus $\mathrm{H}_{\Omega}^m(\varphi)$ is independent of the approximating sequence.\\
Now for the stability under domination, let $\psi \in \mathcal{SH}_m(\Omega, \omega)$ with $\varphi \leq \psi \leq 0$. Let $(\varphi_j) \subset \mathcal{E}^0_m(\Omega, \omega)$ a sequence that decreases to $\varphi$. Define
$
\psi_j := \sup\{\psi, \varphi_j\} \in \mathcal{SH}_m(\Omega, \omega).
$
Since $\varphi_j \in \mathcal{E}^0_m(\Omega, \omega)$ and $\psi \leq 0$, the stability of $\mathcal{E}^0_m(\Omega, \omega)$ under supremum implies that $\psi_j \in\mathcal{E}^0_m(\Omega, \omega)$. Moreover:
$\psi_j \searrow \sup\{\psi, \varphi\} = \psi$ (because $\varphi \leq \psi$)
    and $\psi_j \leq \varphi_j$, so by the comparison principle:
    $$
    \int_\Omega H_m(\psi_j) \leq \int_\Omega H_m(\varphi_j) \leq L.
    $$
Thus $\sup_j \int_\Omega H_m(\psi_j) \leq L < +\infty$, so $\psi \in \mathcal{F}_m(\Omega, \omega)$. Furthermore, by the first part :
$$
\mathrm{H}_{\Omega}^m(\psi) = \lim_j \int_\Omega H_m(\psi_j) \leq L = \mathrm{H}_{\Omega}^m(\varphi).
$$
This completes the proof.
\end{proof}
\begin{defi}
Let $ \chi:\mathbb{R} \to \mathbb{R} $ be a convex increasing function such that $ \chi(t) = t $ for $ t \geq 0 $, and $ \chi(-\infty) = -\infty $.
We define $ \mathcal{E}_\chi^m(\Omega ,\omega) $ to be the set of $ \omega $-$ m $-subharmonic functions $ \varphi \in \mathcal{SH}_m(\Omega ,\omega) $ for which there exists a sequence $ (\varphi_j) \subset \mathcal{E}_0^m(\Omega,\omega) $ such that $ \varphi_j \searrow \varphi $ and
$$
\sup_j \int_\Omega  -\chi(\varphi_j) H_m(\varphi_j) < +\infty.
$$
\end{defi}

\begin{pro}\label{gpc}
Let $ \chi $ be a convex weight function. Then for any $ \varphi, \psi \in \mathcal{E}_0^m(\Omega ,\omega) $ with $ \varphi \leq \psi $, we have:
$$
\int_\Omega  -\chi(\psi) H_m(\psi) \leq 2^m\int_\Omega  -\chi(\varphi) H_m(\varphi).
$$
\end{pro}
\begin{proof}
 Let $T$ be a closed $(\omega,m)$-positive
current of the form $T=(dd^cu_2+\omega)\wedge\dots\wedge(dd^cu_m+\omega)\wedge \omega^{n-m}$, where $u_2, \dots, u_m\in \mathcal{E}^0_m(\Omega,\omega)$.
 Then, by Theorem \ref{IBP}, we can repeat the same proof as in [\cite{GZ}, Lemma 2.3] to obtain $$\int_\Omega-\chi(\varphi)(\omega+dd^c\psi)\wedge T\leq 2\int_\Omega-\chi(\varphi)(\omega+dd^c\varphi)\wedge T.$$  Since $\int_\Omega-\chi(\psi)H_m(\psi)\leq \int_\Omega-\chi(\varphi)H_m(\psi)$, then by taking $T=(\omega+dd^c\varphi)^j\wedge(\omega+dd^c\psi)^{m-j-1}\wedge\omega^{n-m}$ we get the result.
\end{proof}
\begin{thm}\label{thcv}
The complex Hessian operator is well-defined on $ \mathcal{E}_\chi^m(\Omega,\omega) $ and is continuous along decreasing sequences. That is, if $ \varphi_j \searrow \varphi \in \mathcal{E}_\chi^m(\Omega ,\omega) $, then $
H_m(\varphi_j) \longrightarrow H_m(\varphi)$ weakly on  $\Omega .
$
Moreover, for any $ h \in \mathcal{SH}_m(\Omega,\omega) \cap L^\infty(\Omega ) $, we have:
$$
\lim_j \int_\Omega  h H_m(\varphi_j) = \int_\Omega  h H_m(\varphi).
$$
\end{thm}
\begin{proof} The proof of the theorem relies on the techniques developed in \cite{GZ} and \cite{CGZ} (see also \cite[Theorem 3.5]{C}), and is divided into two main steps.\\
\textbf{Step 1.} We claim first that if $\varphi_j \in \mathcal{E}_0^m(\Omega, \omega)$ is a decreasing sequence converging to $\varphi \in \mathcal{E}^m_\chi(\Omega, \omega)$, then the sequence of Hessian measures $\mu_j := (dd^c \varphi_j+\omega)^m\wedge \omega^{n-m}$ converges weakly to a Radon measure $\mu$, and $H_m(\varphi^k)$ converges to the same limit $\mu$, with $\varphi^k=\max\{\varphi, -k\}$.\\
 Fix a smooth function $\theta \in C_c^\infty(\Omega)$ with compact support. For each $k > 0$, define the truncations $\varphi_j^k := \max(\varphi_j, -k)$ and $\varphi^k := \max(\varphi, -k)$. Then $\varphi_j^k$ is a decreasing sequence of bounded $\omega$-psh functions converging to $\varphi^k$.
Now fix $k > 0$. Since the complex Hessian operator is continuous on decreasing sequences of bounded $\omega$-$m$-subharmonic functions, the sequence of measures $\mu_j^k := H_m(\varphi_j^k)$ converges weakly to $\mu^k := H_m(\varphi^k)$. Therefore, for any $\theta \in C_c^\infty(\Omega)$, the sequence
$
\left\{ \int_\Omega \theta \, (\omega + dd^c \varphi_j^k)^n \right\}_j
$
converges and is hence a Cauchy sequence.
Hence, for any $\varepsilon > 0$, there exists $N_k \in \mathbb{N}$ such that for all $j, \ell \geq N_k$,
$$
\left| \int_\Omega \theta \, H_m( \varphi_j^k) - \int_\Omega \theta \, H_m( \varphi_\ell^k) \right| < \varepsilon.
$$
We next estimate the difference between $\int_\Omega \theta \, d\mu_j$ and $\int_D \theta \, d\mu_j^k$:
$$
\left| \int_\Omega \theta \, d\mu_j - \int_\Omega \theta \, d\mu_j^k \right| = \left| \int_{\{\varphi_j \leq -k\}} \theta \, d\mu_j - \int_{\{\varphi_j \leq -k\}} \theta \, d\mu_j^k \right| \leq \|\theta\|_\infty \, \mu_j(\{ \varphi_j \leq -k \}).
$$
Using the energy control and monotonicity of $\chi$, we have
$$
\mu_j(\{\varphi_j \leq -k\}) \leq \frac{1}{-\chi(-k)} \int_\Omega -\chi(\varphi_j) \, d\mu_j \leq \frac{C}{-\chi(-k)}.
$$
Thus,
$$
\left| \int_\Omega \theta \, d\mu_j - \int_\Omega \theta \, d\mu_j^k \right| \leq \frac{C\|\theta\|_\infty}{-\chi(-k)}.
$$
Similarly for $\mu_\ell$, so for $j, \ell \geq N_k$ we get
\begin{align*}
\left| \int_\Omega \theta \, d\mu_j - \int_\Omega \theta \, d\mu_\ell \right|
&\leq \left| \int_\Omega \theta \, d\mu_j - \int_\Omega \theta \, d\mu_j^k \right|
+ \left| \int_\Omega \theta \, d\mu_j^k - \int_\Omega \theta \, d\mu_\ell^k \right|
+ \left| \int_\Omega \theta \, d\mu_\ell^k - \int_\Omega \theta \, d\mu_\ell \right| \\
&< \frac{2C\|\theta\|_\infty}{-\chi(-k)} + \varepsilon.
\end{align*}
Choosing $k$ large so that $\frac{2C\|\theta\|_\infty}{-\chi(-k)} < \varepsilon$, we conclude that $\left\{\int_\Omega \theta \, d\mu_j\right\}_j$ is a Cauchy sequence and hence it converges weakly to a Radon measure $\mu$\\
Now, since
$$
\left| \int \theta \, H_m( \varphi_j) - \int \theta \, H_m(\varphi_j^k) \right| \leq \|\theta\|_\infty \, \mu_j(\{\varphi_j \leq -k\}) \leq \frac{C}{-\chi(-k)}.
$$
By the continuity of the Hessian operator on bounded $\omega$-$m$-subharmonic functions, letting $j \to \infty$, we obtain $
H_m( \varphi^k) \longrightarrow \mu$ as $k \to \infty.$
Also, for each $k$ we have
$
-\chi(\varphi_j^k) \leq -\chi(\varphi_j).
$
Hence by Proposition \ref{gpc} 
$$
\int -\chi(\varphi_j^k) H_m(\varphi_j^k) 
\leq 2^m \int -\chi(\varphi_j)H_m(\varphi_j).
$$
Passing to the limit as $j \to \infty$ and using monotone convergence, we deduce that:
$$
\sup_k \int -\chi(\varphi^k) H_m(\varphi^k) < +\infty.
$$
\noindent\textbf{Step 2.} 
Let $\psi_j \searrow \varphi$ be another decreasing sequence in $\mathcal{E}_0^m(\Omega, \omega)$. By Step 1, $\nu_j:=H_m( \psi_j)$ converges weakly to a Radon measure $\nu$.
For each $j$, since $\psi_j$ is bounded above and converges to $\varphi$, there exists $k_j$ such that $
\psi_j \geq \varphi^{k_j} := \max(\varphi, -k_j).
$ Hence $
-\chi(\psi_j) \leq -\chi(\varphi^{k_j}).
$
It follows by Proposition \ref{gpc} that
$$
\int -\chi(\psi_j) H_m(\psi_j)
\leq 2^m \int -\chi(\varphi^{k_j}) H_m( \varphi^{k_j}),
$$
and thus the energy condition holds uniformly for $\psi_j$. 
To prove that $\mu = \nu$, we fix $\theta \in C_c^\infty(D)$ and $k > 0$.
Then:
$$
\int \theta H_m(\varphi_j^k) \to \int \theta H_m( \varphi^k),
\quad 
\int \theta H_m( \psi_j^k) \to \int \theta H_m(\varphi^k).
$$
We estimate:
\begin{align*}
\left| \int \theta \, d\mu_j - \int \theta \, d\nu_j \right| 
&\leq \left| \int \theta \, (H_m(\varphi_j) - H_m( \varphi_j^k))\right| \\
&+ \left| \int \theta \, (H_m(\varphi_j^k) - H_m(\varphi^k)) \right| \\
&+ \left| \int \theta \, (H_m(\psi_j^k) - H_m( \varphi^k)) \right| \\
&+ \left| \int \theta \, (H_m(\psi_j^k) - H_m( \psi_j)) \right|.
\end{align*}
Each term is small for large $j$ and $k$, due to the following  estimate
$$
\left| \int \theta \, \Big(H_m( \varphi_j) - H_m(\varphi_j^k)\Big) \right| 
\leq \|\theta\|_\infty \, \mu_j(\{ \varphi_j \leq -k \}) 
\leq \frac{C}{-\chi(-k)},
$$
and similarly for $\psi_j$. Therefore,
$$
\lim_j \int \theta \, d\mu_j = \lim_j \int \theta \, d\nu_j = \int \theta \, d\mu,
$$
which implies that $\mu = \nu$.\\ Now to prove that for any $ h \in \mathcal{SH}_m(\Omega,\omega) \cap L^\infty(\Omega ) $, we have $
\lim_j \int_\Omega  h H_m(\varphi_j) = \int_\Omega  h H_m(\varphi),
$
 we can proceed with the same argument as above.

\end{proof}
The proposition below follows by applying the techniques used in the proof of Theorem \ref{thcv}, together with \cite[Theorem 2.1]{CGZ} and \cite[Theorem 2.6]{GZ}.
\begin{pro}
Let $ \varphi \in \mathcal{SH}_m(\Omega ,\omega) $. Suppose  that there exists a decreasing sequence $ (\varphi_j) \subset \mathcal{E}_0^m(\Omega ,\omega) $ with $ \varphi_j \searrow \varphi $ and $
\sup_j \int_\Omega  -\chi(\varphi_j)H_m(\varphi_j) < +\infty.
$
Then $ \varphi \in \mathcal{E}_\chi^m(\Omega ,\omega) $, and we have:
$$
\lim_{j \to \infty} \int_\Omega  -\chi(\varphi_j) H_m(\varphi_j) = \int_\Omega  -\chi(\varphi) H_m(\varphi).
$$
\end{pro}

\section{Subextension of $\omega$-$m$-subharmonic functions}
We start this section by proving the existence of a subextension for functions in 
$\mathcal{F}_m(\Omega,\omega)$.
\begin{thm}\label{ExtFm}
Let $\Omega \subset X$ be a quasi-hyperconvex domain satisfying $\int_\Omega\omega^n<\int_X\omega$ and $\varphi \in \mathcal{F}_m(\Omega,\omega)$ such that $\mathrm{H}_{\Omega}^m(\varphi) \leqslant \int_X \omega^n$. Then there exists a function $\tilde{\varphi} \in \mathcal{SH}_m(X,\omega)$ such that $\tilde{\varphi} \leqslant \varphi$ on $\Omega$.
\end{thm}

\begin{proof}
Let $(\varphi_j)\subset \mathcal{E}_0^m(\Omega,\omega)$ be a decreasing sequence which converges to $\varphi$ on $\Omega$. By Lemma \ref{fcomp} we have
$$
\int_\Omega H_m(\varphi_j) \leq \mathrm{H}_{\Omega}^m(\varphi).
$$
Put $\mu_j=\mathbf{1}_\Omega H_m(\varphi_j) + \varepsilon_j \omega^n$ with $0\leq \varepsilon_j<1$  chosen such that $\mu_j(X)=\int_X\omega^n$. Since $\mu_j$ does not charge $m$-polar sets, it follows from \cite[Theorem 1.3]{CN}, that there exists $u_j \in \mathcal{E}(X,\omega,m)$ such that $
H_m(u_j) = \mu_j$ and $\sup_X u_j = -1$. Let $v_j=\max\{u_j,\varphi_j\}$ on $\Omega$. It follows from Proposition \ref{demailly} that  $$H_m(v_j)\geq {\bf{1}}_{\{u_j>\varphi_j\}}H_m(u_j)+{\bf{1}}_{\{u_j\leq\varphi_j\}}H_m(\varphi_j)\geq H_m(\varphi_j).$$
Since $\liminf(\varphi_j-v_j)\geq 0$ on $\partial\Omega$, then by Theorem \ref{pcf} we get $\varphi_j\geq u_j$ on $\Omega$.  Take \(u := (\limsup_{j \to +\infty} u_j)^*\). Then \(u \in \mathcal{SH}_m(X,\omega)\) and satisfies $u \leqslant \varphi$ on $\Omega$ and $\max_X u = -1$.\end{proof}
From this theorem, we deduce that the envelope 
$\tilde{\varphi}:=\sup\{u\in\mathcal{SH}^-_m(X,\omega);\; u\leq \varphi \; \mbox{on}\; \Omega\}$
  exists for functions  $\varphi \in \mathcal{F}_m(\Omega,\omega)$ such that $\mathrm{H}_{\Omega}^m(\varphi) \leqslant \int_X \omega^n$.

\begin{lem}\label{extE0} 
Let $\Omega$ be a quasi $m$-hyperconvex domain and $\varphi \in \mathcal{E}_0^m(\Omega, \omega)$ be such that $\int_\Omega H_m(\varphi) \leq \int_X \omega^n$, then $\tilde{\varphi} \in \mathcal{SH}_m(X, \omega) \cap L^\infty(X)$ and $\mathbf{1}_\Omega H_m(\tilde{\varphi})\leq \mathbf{1}_\Omega H_m(\varphi)$ in the sense of measures on $X$. Moreover the measure $H_m(\tilde{\varphi})$ is carried by the Borel set $\{x \in \bar{\Omega}; \tilde{\varphi}(x) = \varphi(x)\}$.
\end{lem}
\begin{proof}
{\bf Step 1.} Suppose that $\varphi \in \mathcal{E}_0^m(\Omega, \omega)\cap C^0(\bar{\Omega})$ with $\varphi=0$ on $\partial\Omega$. Then $D=[X\setminus \bar{\Omega}]\cup[\{\tilde{\varphi}<\varphi\}\cap\bar{\Omega}]$ is open. By Proposition \ref{DP} and the balayage technique, it follows that $H_m(\tilde{\varphi})=0$ in $D$ 
establishing the inequality $\mathbf{1}_\Omega H_m(\tilde{\varphi})\leq \mathbf{1}_\Omega H_m(\varphi)$ in this region and the measure $H_m(\tilde{\varphi})$ is carried by the Borel set $\{x \in \bar{\Omega}; \tilde{\varphi}(x) = \varphi(x)\}$.
As in the local theory (see \cite{CH}), to extend this inequality to the coincidence set $A = \{z \in \Omega : \varphi(z) = \tilde{\varphi}(z)\}$, consider an arbitrary compact subset $K \subset \Omega\cap A$. Note that $K \subset \{\tilde{\varphi} > \varphi- \varepsilon\}$ for sufficiently small $\varepsilon > 0$, then by Proposition \ref{demailly}, we derive the estimate:
$$\begin{array}{ll}
\displaystyle\int_K  H_m(\tilde{\varphi}) &= \displaystyle\int_K {\bf{1}}_{\{\tilde{\varphi} > \varphi - \varepsilon\}}  H_m(\tilde{\varphi}) \\
&= \displaystyle\int_K {\bf{1}}_{\{\tilde{\varphi} > \varphi - \varepsilon\}} H_m(\max\{\hat{\varphi}, \varphi - \varepsilon\}) \\
& \leq \displaystyle\int_K H_m(\max\{\tilde{\varphi}, \varphi - \varepsilon\}).
\end{array}
$$
Since  $\max\{\tilde{\varphi}, \varphi - \varepsilon\} \searrow \varphi$ as $\varepsilon \to 0$, then
$
H_m(\max\{\tilde{\varphi}, \varphi - \varepsilon\})$ converges weakly to $H_m(\varphi).$
Since $\chi_K$ is upper semicontinuous, we may approximate it pointwise by a decreasing sequence of continuous functions $\{\phi_j\}_j$ with uniform upper bound. Applying Lebesgue's dominated convergence theorem yields the crucial estimate:
$$
\limsup_{\varepsilon \to 0} \int_{\Omega} \chi_K H_m\max\{\tilde{\varphi}, \varphi - \varepsilon\}) 
\leq \inf_{j\in\mathbb{N}} \left( \lim_{\varepsilon \to 0} \int_{\Omega} \phi_j H_m(\max\{\hat{\varphi}, \varphi - \varepsilon\}) \right)
= \inf_{j\in\mathbb{N}} \int_{\Omega} \phi_j H_m(\varphi) 
= \int_{\Omega} \chi_K H_m(\varphi).
$$
Combining this estimate with the previous arguments completes the proof of step 1.\\
{\bf Step 2.}   Assume now that $\varphi \in \mathcal{E}_m^0(\Omega,\omega)$. By Proposition \ref{ccv} there exists  a decreasing sequence $(\varphi_j)\subset \mathcal{E}_0^m(\Omega, \omega)\cap C^0(\bar{\Omega})$ that is converging pointwise to $\varphi$. The corresponding sequence of subextensions $(\tilde{\varphi}_j)$ then decreases  to $\tilde{\varphi}\in \mathcal{SH}_m(X,\omega)\cap L^\infty(X)$. By Step 1 and the continuity of the Hessian measure with respect to decreasing sequences, we have $\mathbf{1}_\Omega H_m(\tilde{\varphi})\leq \mathbf{1}_\Omega H_m(\varphi)$ in the sense of measures on $X$. Moreover, since the sequence $(\tilde{\varphi}_j)_j$ is uniformly bounded (\(\inf_X\tilde{\varphi}\leq\tilde{\varphi}_j\leq\sup_X\tilde{\varphi}_1\) ), then the sequence of Hessian measures $\mu_j := H_m(\tilde{\varphi}_j)$ is uniformly controlled by the $m$-capacity.
By Step 1 we have 
$
\int_{\{\tilde{\varphi}_j < \varphi_j\}} H_m(\tilde{\varphi}_j) = 0$ for all $j \in \mathbb{N}.$
 It follows from Lemma \ref{lemconv} that  for any $s \geq 0$:
$$
\int_{\{\tilde{\varphi}_s < \varphi\}} H_m(\tilde{\varphi}) 
\leq \liminf_{j\to\infty} \int_{\{\tilde{\varphi}_s < \varphi\}} H_m(\tilde{\varphi}_j)
\leq \liminf_{j\to\infty} \int_{\{\tilde{\varphi}_j < \varphi_j\}} H_m(\tilde{\varphi}_j) 
= 0,
$$
 Taking the limit as $s \to +\infty$ we get $supp H_m(\tilde{\varphi})\subset\overline{\{x \in \bar{\Omega}; \tilde{\varphi}(x) = \varphi(x)\}}$. But by the quasi-continuity and since \(H_m(\tilde{\varphi})\) does not charge $m$-polar sets, the measure $H_m(\tilde{\varphi})$ is carried by the Borel set $\{x \in \bar{\Omega}; \tilde{\varphi}(x) = \varphi(x)\}$.
\end{proof}

It follows from the above theorem that given $\varphi \in \mathcal{F}_m(\Omega,\omega)$ such that $\mathrm{H}_\Omega^m(\varphi) \leqslant \int_X \omega^n$, the set $\{\psi \in \mathcal{SH}_m(X,\omega); \psi \leqslant \varphi \text{ on } \Omega\}$ is not empty. Hence the following function
$
\tilde{\varphi} = \tilde{\varphi}_\Omega := \sup\{\psi \in \mathcal{SH}_m(X,\omega); \psi \leqslant \varphi \text{ on } \Omega\}
$
is a well-defined $\omega$-$m$-subharmonic function on $X$ and will be called the \textit{maximal subextension} of $\varphi$ from $\Omega$ to $X$.

\begin{pro}
Let $\varphi \in \mathcal{F}_m(\Omega, \omega)$ and consider a decreasing approximating sequence $(\varphi_j) \subset \mathcal{E}_0^m(\Omega, \omega)$ converging to $\varphi$. The corresponding sequence of maximal subextensions $(\widetilde{\varphi}_j)$ exhibits the following properties:
\begin{enumerate}
    \item It decreases monotonically to $\widetilde{\varphi}$ throughout $X$
    \item Any  Borel measure $\mu$ on $X$ which is a limit point of the sequence of measures $(H_m(\widetilde{\varphi}_j))$ satisfies the domination
   $
    \mathbf{1}_\Omega \mu \leq \mathbf{1}_\Omega H_m(\varphi)
    $
    in the sense of Radon measures on $X$.
\end{enumerate}
\end{pro}
\begin{proof}
The inequality $\widetilde{\varphi} \leq \widetilde{\varphi}_j$ and monotonicity of $(\widetilde{\varphi}_j)$ are immediate from construction. Define $\psi := \lim_{j\to\infty} \widetilde{\varphi}_j$. Observe that $\psi \in \mathcal{SH}_m(X,\omega)$ and $\psi \geq \widetilde{\varphi}$. Crucially, the containment $\psi \leq \widetilde{\varphi}_j \leq \varphi_j$ on $\Omega$ for all $j$ implies  that $\psi \leq \varphi$ on $\Omega$ through passage to the limit. Thus $\psi$ constitutes a subextension of $\varphi$ to $X$, necessitating $\psi \leq \widetilde{\varphi}$. We conclude that $\psi \equiv \widetilde{\varphi}$ on $X$.
By Lemma \ref{extE0}, we have $
\mathbf{1}_\Omega H_m(\widetilde{\varphi}_j) \leq \mathbf{1}_\Omega H_m(\varphi_j)
$
for each $j$. Monotone convergence of Hessian measures under decreasing sequences yields the desired inequality in the statement.
\end{proof}
Provided the given function has finite weighted Hessian energy, the maximal subextension is shown to admit a well-defined Hessian measure:
\begin{thm}\label{extEchi}
Let $\Omega \subset X$ be a quasi-$m$-hyperconvex domain with $\int_\Omega \omega^n < \int_X \omega^n$, $\chi : \mathbb{R} \mapsto \mathbb{R}$ be a convex weight function and  $\varphi \in \mathcal{E}_\chi^m(\Omega, \omega)$ satisfy 
$
\int_\Omega H_m(\varphi) \leq \int_X \omega^n.
$
Then the maximal $\omega$-$m$-subharmonic subextension $\tilde{\varphi}$ of $\varphi$ to $X$ exists and satisfies:
\begin{enumerate}
\item[(i)] $\tilde{\varphi} \in \mathcal{E}_\chi^m(X, \omega)$, and 
$
\int_X -\chi(\tilde{\varphi})  H_m(\tilde{\varphi}) \leq \int_\Omega -\chi(\varphi)  H_m(\varphi).
$
\item[(ii)] The Hessian measure satisfies 
$
\mathbf{1}_\Omega H_m(\tilde{\varphi}) \leq \mathbf{1}_\Omega H_m(\varphi)
$
in the sense of Borel measures on $X$.
\item[(iii)] The measure $H_m(\tilde{\varphi})$ is supported in the set 
$
\operatorname{supp} H_m(\tilde{\varphi}) \subset \{ \tilde{\varphi} = \varphi \} \cup \partial \Omega.
$
\end{enumerate}
\end{thm}
\begin{proof}
Let $(\varphi_j)$ be a sequence in $\mathcal{E}_0^m(\Omega, \omega)$ decreasing to $\varphi$. For each $j$, denote by $\tilde{\varphi}_j$ the maximal subextension of $\varphi_j$ from $\Omega$ to $X$. By Lemma \ref{extE0}, each $\tilde{\varphi}_j \in \mathcal{SH}_m(X, \omega) \cap L^\infty(X)$,  the measure $H_m(\tilde{\varphi}_j)$ is supported on the contact set $A=\{x\in\bar{\Omega};\;\, \tilde{\varphi}_j(x) = \varphi_j(x)\}$ and $
\mathbf{1}_\Omega H_m(\tilde{\varphi}_j) \leq \mathbf{1}_\Omega H_m(\varphi_j)
$. Since \(\chi(\tilde{\varphi}_j)=\chi(\varphi_j)=\chi(0)=0\) on \(\partial\Omega\cap A\), then for every $j$,
$$
-\chi(\tilde{\varphi}_j) H_m(\tilde{\varphi}_j) \leq  -\chi(\varphi_j)\mathbf{1}_\Omega H_m(\varphi_j),
$$
in the sense of measures. This implies that there exists a constant $C > 0$ independent of $j$ such that:
$$
\int_X -\chi(\tilde{\varphi}_j) H_m(\tilde{\varphi}_j) \leq \int_\Omega -\chi(\varphi_j) H_m(\varphi_j) \leq C.
$$
Since $\tilde{\varphi}_j \searrow \tilde{\varphi}$ on $X$, the convergence theorem in \cite{CN} implies that $\tilde{\varphi} \in \mathcal{E}_\chi^m(X, \omega)$, and 
$$
 -{\chi(\tilde{\varphi})}  \mathbf{1}_\Omega H_m(\tilde{\varphi}) \leq -{\chi(\varphi})\mathbf{1}_\Omega  H_m(\varphi).
$$
For (iii), consider the truncations $\varphi_j := \max\{ \varphi, -j \}$. Fix $t > 0$ and note that the function $(1 + \max\{ \varphi/t, -1 \})$ vanishes on $\{ \varphi \leq -t \}$ and is bounded above by $1$. Moreover, for $j > t$, we have $\{ \varphi > -t \} \subset \{ \varphi > -j \}$ and 
$$
\mathbf{1}_{\{ \varphi > -j \}} H_m(\varphi_j) \nearrow \mathbf{1}_{\{ \varphi > -\infty \}} H_m(\varphi),
$$
as shown in \cite{CN}. Hence, the sequence of measures 
$
\mu_j := \big(1 + \max\{\varphi/t, -1\}\big) H_m(\varphi_j)
$
is increasing and satisfies 
$
\mu_j \leq \mathbf{1}_{\{ \varphi > -\infty \}} H_m(\varphi).
$
Then, for fixed $s$ and $t$, by Lemma \ref{lemconv} we have
\begin{align*}
\int_{\{ \tilde{\varphi}_s < \varphi \}} \big(1 + \max\{\varphi/t, -1\}\big) H_m(\varphi) 
&\leq \liminf_j \int_{\{ \tilde{\varphi}_s < \varphi \}} \big(1 + \max\{\varphi/t, -1\}\big) H_m(\tilde{\varphi}_j) \\
&\leq \liminf_j \int_{\{ \tilde{\varphi}_j < \varphi_j \}} H_m(\tilde{\varphi}_j) = 0.
\end{align*}
Letting $t \to \infty$, we observe that $\big(1 + \max\{\varphi/t, -1\}\big) \nearrow \mathbf{1}_{\{ \varphi > -\infty \}}$, so we obtain:
$
\int_{\{ \tilde{\varphi}_s < \varphi \}} H_m(\tilde{\varphi}) = 0.
$
Finally, letting $s \to \infty$, and by quasi-continuity together with the fact that \(H_m(\tilde{\varphi})\) does not charge $m$-polar set, we obtain the desired conclusion:
$
H_m(\tilde{\varphi})$ is supported on $\{ \tilde{\varphi} = \varphi \} \cup \partial \Omega.$
\end{proof}
\begin{rem}
Unlike the local case, the Hessian measure \( H_m(\tilde{\varphi}) \) can charge the boundary of \(\Omega\). To see this, choose \(\varphi\) such that
\(
\int_\Omega H_m(\varphi) < \int_X \omega^n = \int_X H_m(\tilde{\varphi}),
\)
where the second equality follows from the condition \(\tilde{\varphi} \in \mathcal{E}_m(X,\omega)\). Then \( H_m(\tilde{\varphi}) \) must necessarily give positive mass to \(\partial\Omega\).
Indeed, from the inequality \(\mathbf{1}_\Omega H_m(\tilde{\varphi}) \le \mathbf{1}_\Omega H_m(\varphi)\), we can decompose
\(
H_m(\tilde{\varphi}) = f \cdot \mathbf{1}_\Omega H_m(\varphi) + \nu,
\)
with \(0 \le f \le 1\) and \(\operatorname{supp}(\nu) \subset \partial\Omega\). Consequently,
\(
\nu(\partial\Omega)= \int_X H_m(\tilde{\varphi}) - \int_\Omega f \cdot H_m(\varphi) \geq \int_X H_m(\tilde{\varphi}) - \int_\Omega H_m(\varphi) > 0.
\)
 \end{rem}

\end{document}